\documentclass[12pt]{article}

\usepackage{amssymb,amsmath,amsthm}
\usepackage{graphicx, epstopdf, caption, subcaption}

      \theoremstyle{plain}
      \newtheorem{theorem}{Theorem}[section]
      \newtheorem{lemma}{Lemma}[section]

      \theoremstyle{definition}
      \newtheorem{definition}{Definition}[section]
      
      \theoremstyle{remark}
      \newtheorem{remark}{Remark}[section]

\def\st{\raisebox{-3pt}{\Huge $*$}}

\begin{document}

\title{Embeddings of $\st$-graphs into 2-surfaces}
\author{Tyler Friesen
\footnote{Ohio State University, 154 W 12th Ave., Columbus, Ohio 43210}
\footnote{\tt {friesen.15@osu.edu}}
\and Vassily Olegovich Manturov
\footnote{Peoples' Friendship University of Russia, Moscow 117198,
Ordjonikidze St., 3} \footnote{\tt {vomanturov@yandex.ru}}
\footnote{Partially supported by grants of the Russian Government
11.G34.31.0053, RF President NSh � 1410.2012.1, Ministry of
Education and Science of the Russian Federation 14.740.11.0794.}}

\maketitle

\begin{abstract}
This paper considers $\st$-graphs in which all vertices have degree 4 or 6, and studies the question of calculating the genus of orientable 2-surfaces into which such graphs may be embedded. A $\st$-graph is a graph endowed with a formal adjacency structure on the half-edges around each vertex, and an embedding of a $\st$-graph is an embedding under which the formal adjacency relation on half-edges corresponds to the adjacency relation induced by the embedding. $\st$-graphs are a natural generalization of four-valent framed graphs, which are four-valent graphs with an opposite half-edge structure. In \cite{manturov2}, the question of whether a four-valent framed graph admits a $\mathbb Z_2$-homologically trivial embedding into a given surface was shown to be equivalent to a problem on matrices. We show that a similar result holds for $\st$-graphs in which all vertices have degree 4 or 6. This gives an algorithm in quadratic time to determine whether a $\st$-graph admits an embedding into the plane.
\end{abstract}

{\bf Keywords:} Graph, $\st$-graph, surface, embedding, genus

{\bf AMS Subject Classification:} Primary 05C10; Secondary 57C15, 57C27

\section{Introduction}

\begin{definition}
A \emph{four-valent framed graph} is a regular four-valent graph which at each vertex is endowed with a pairing structure on the four outgoing half-edges. Half-edges which are paired together by this structure are (formally) \emph{opposite}.
\end{definition}

\begin{definition}
An \emph{angle} in a four-valent framed graph is a pair of non-opposite half-edges at a vertex.
\end{definition}

\begin{remark}
By an embedding of a four-valent framed graph $\Gamma$ into a surface $S$ we always mean an embedding of $\Gamma$ into $S$ such that the formal relation of being opposite coincides with the relation of being opposite induced by the embedding.
\end{remark}

\begin{definition}
A \emph{$\st$-graph} is a graph which at each vertex has a bijection from the outgoing half-edges to the vertices of a cycle graph. Half-edges which are mapped to adjacent vertices are (formally) \emph{adjacent}. Half-edges are said to be \emph{opposite} if they are mapped to vertices of maximal distance in the cycle graph.
\end{definition}

\begin{remark}
By an embedding of a $\st$-graph $\Gamma$ into a surface $S$ we always mean an embedding of  $\Gamma$ into $S$ such that the formal relation of being adjacent coincides with the relation of being adjacent induced by the embedding.
\end{remark}

\begin{definition}
An \emph{angle} in a $\st$-graph is a pair of adjacent half-edges at a vertex.
\end{definition}

\begin{remark}
For any 4-regular $\st$-graph $\Gamma$, we can view $\Gamma$ as a four-valent framed graph by calling two edges at a vertex \emph{opposite} if and only if they are not adjacent.
\end{remark}

\begin{definition}
A \emph{checkerboard embedding} of a framed four-valent graph or $\st$-graph $\Gamma$ into $S$ is an embedding such that the cells of $S \setminus \Gamma$ admit a 2-coloring under which any cells with a common edge have different colors. 
\end{definition}

\begin{remark}
Checkerboard embeddings are exactly those embeddings whose first $\mathbb Z_2$-homology class is zero.
\end{remark}

In \cite{manturov1} the first named author (V.O.M.) gave a solution to the question of whether four-valent framed graphs are planar. In \cite{manturov2}, he addressed the question of determining the genus of surfaces into which four-valent framed graphs can be embedded, in particular considering the special case of surfaces into which four-valent framed graphs may be checkerboard-embedded. In \cite{friesen}, the second named author (T.F.) gave a planarity criterion for $\st$-graphs.
The goal of this paper is to provide a method for determining whether a $\st$-graph $\Gamma$ has a checkerboard embedding into an orientable surface of genus $g$. In Theorem \ref{main} we show that this is equivalent to a problem on matrices.

The authors of this paper would like to thank Victor Anatolievich Vassiliev and Sergei Vladimirovich Chmutov for valuable discussions.

\section{Basic Notions}

\begin{definition}
An \emph{atom} \cite{fomenko} is a closed 2-surface $S$ into which a 4-valent graph $\Gamma$ (the \emph{skeleton} of the atom) is embedded in such a way that it divides $S$ into black and white cells so that cells sharing an edge have different colors.
\end{definition}

Note that the atom induces a framing on its skeleton. Opposite half-edges at each vertex of the skeleton determine four angles among which two pairs of opposite angles are specified. One of the pairs of opposite angles must be black, and the other white. Given a four-valent framed graph $\Gamma$, the atoms corresponding to $\Gamma$ are uniquely determined by a choice of one of the two possible colorings at each vertex.

\begin{definition}
A \emph{$\st$-atom} is a closed 2-surface $S$ into which a connected graph $\Gamma$ (the \emph{skeleton} of the $\st$-atom) is embedded in such a way that it divides $S$ into black and white cells so that cells sharing an edge have different colors.
\end{definition}

This embedding induces a $\st$-structure on the skeleton. The $\st$-structure at each vertex determines a set of $d$ angles among which we say that two angles are adjacent if they share a half-edge. Two adjacent angles never have the same color. Thus the angles around a vertex can be partitioned into two sets $A_1$ and $A_2$ in such a way that for any $\st$-atom corresponding to the $\st$-graph $\Gamma$, either all angles in $A_1$ are black and all angles in $A_2$ are white, or all angles in $A_1$ are white and all angles in $A_2$ are black. Thus given a connected $\st$-graph $\Gamma$, the $\st$-atoms corresponding to $\Gamma$ are uniquely determined by a choice of one of the two possible colorings at each vertex. Thus the main problem can be reformulated as follows:

Given a $\st$-graph $\Gamma$ in which all vertices have degree 4 or 6, choose a coloring for the angles around each vertex such that the genus of the resulting atom is minimal. If such a graph has $n$ vertices, there are $2^n$ corresponding $\st$-atoms.

\begin{definition}
A framed four-valent graph satisfies the \emph{source-sink condition} \cite{manturov2} if each edge of it can be endowed with an orientation in such a way that for each vertex, two opposite edges are emanating, and the remaining two are incoming.
\end{definition}

\begin{definition}
A  $\st$-graph satisfies the \emph{source-sink condition} if each edge of it can be endowed with an orientation in such a way that for any two adjacent edges at a vertex, one edge is emanating and the other is incoming.
\end{definition}

\begin{lemma}
A $\st$-graph satisfies the source-sink condition if and only if it is checkerboard-embeddable into an orientable surface.
\end{lemma}

\begin{remark}
This result is well-known for framed four-valent graphs; see for example \cite{manturov2}
\end{remark}

\begin{proof}
Suppose that a $\st$-graph $\Gamma$ is checkerboard-embedded into an orientable surface $S$. We can orient the edges in such a way that the edges around any white cell are oriented clockwise around that cell, showing that $\Gamma$ meets the source-sink condition. Conversely, if a $\st$-graph $\Gamma$ is checkerboard-embedded into a surface $S$, we can orient the cells of the embedding in such a way that in white cells, the edges around the cell are oriented clockwise, and in black cells, the edges around the cell are oriented counterclockwise.
\end{proof}

\begin{remark}
If $\Gamma$ is a connected $\st$-graph with each vertex of valency 4 or 6, and $\Gamma$ does not satisfy the source-sink condition, then there exists a double covering of $\Gamma$ which satisfies the source-sink condition. Thus every question concerning the embedding of a $\st$-graph with each vertex of valency 4 or 6 can be reduced to the case where the source-sink condition is satisfied.
\end{remark}

\begin{definition}
An \emph{Euler circuit} $C$ of a $\st$-graph $\Gamma$ is a surjective mapping $S^1 \to \Gamma$ which is one-to-one except at the vertices, and such that every vertex of degree $d$ has $\frac{d}{2}$ preimages.
\end{definition}

\begin{definition}
Given an Euler circuit $C$ of a $\st$-graph $\Gamma$, a 4-vertex $v \in \Gamma$ is \emph{rotating} if for every $a \in C^{-1}(\{v\})$, $C(a+\epsilon)$ and $C(a-\epsilon)$ are on adjacent half-edges around $v$.
\end{definition}

\begin{definition}
Given an Euler circuit $C$ of a $\st$-graph $\Gamma$, a 6-vertex $v \in \Gamma$ is \emph{rotating} if for every $a \in C^{-1}(\{v\})$, $C(a+\epsilon)$ and $C(a-\epsilon)$ are on adjacent half-edges around $v$.
\end{definition}

\begin{definition}
Given an Euler circuit $C$ of a $\st$-graph $\Gamma$, a 6-vertex $v \in \Gamma$ is \emph{splitting} if for some $a \in C^{-1}(\{v\})$, $C(a+\epsilon)$ and $C(a-\epsilon)$ are on opposite half-edges around $v$, and for the other two points $b, c \in C^{-1}(\{v\})$, $C(b+\epsilon)$ and $C(b-\epsilon)$ are on adjacent half-edges, and $C(c+\epsilon)$ and $C(c-\epsilon)$ are on adjacent half-edges.
\end{definition}

\begin{figure}
	\begin{subfigure}{0.3\textwidth}
	\centering
	\includegraphics[trim = 8cm 20cm 8cm 3cm, clip=true, totalheight=3cm]{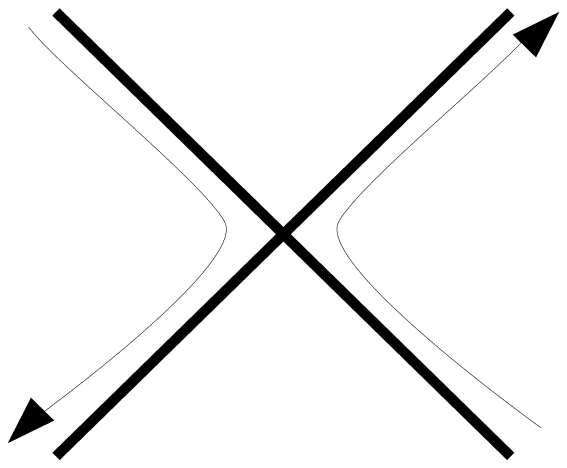}
	\caption{A rotating vertex of degree 4}
	\end{subfigure}
	\quad
	\begin{subfigure}{0.3\textwidth}
	\centering
	\includegraphics[trim = 5cm 17.5cm 3cm 5cm, clip=true, totalheight=3cm]{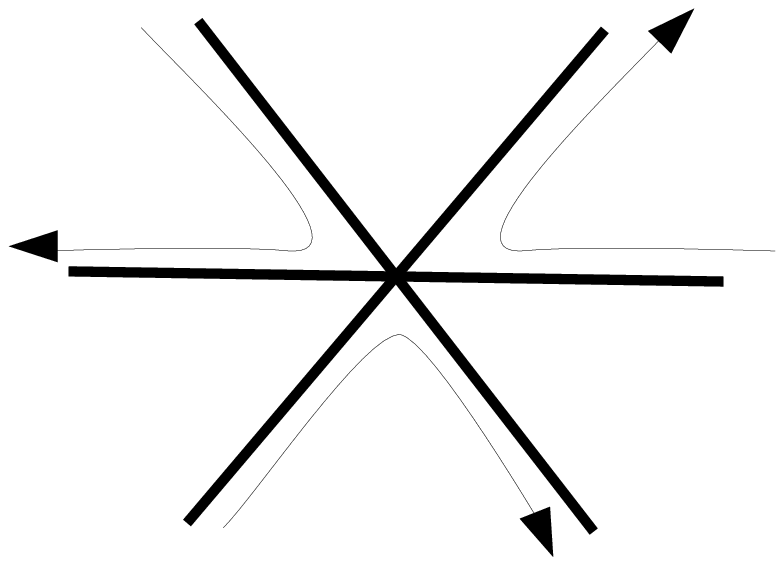}
	\caption{A rotating vertex of degree 6}
	\end{subfigure}
	\quad
	\begin{subfigure}{0.3\textwidth}
	\centering
	\includegraphics[trim = 7cm 17cm 0cm 3.8cm, clip=true, totalheight=3cm]{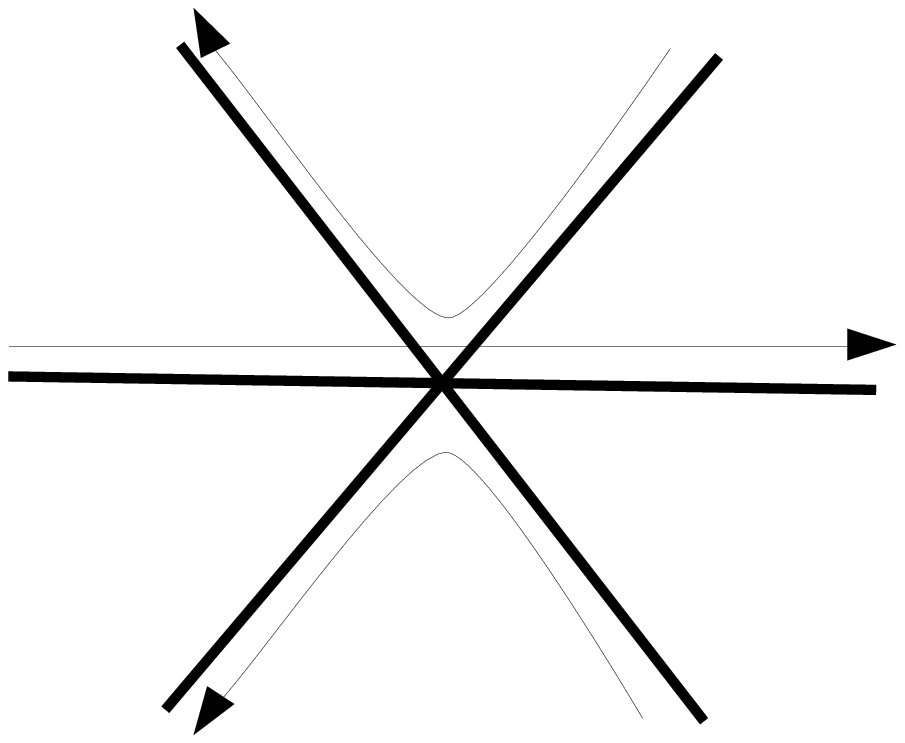}
	\caption{A splitting vertex of degree 6}
	\end{subfigure}
	\caption{}
\end{figure}

\begin{definition}
A \emph{rotating-splitting circuit} is an Euler circuit on a $\st$-graph with all vertices of degree 4 or 6, for which every 4-vertex is rotating and every 6-vertex is rotating or splitting.
\end{definition}

\begin{lemma}\label{rscycle-exists}
Let $\Gamma$ be a connected $\st$-graph in which all vertices have degree 4 or 6. Then if
$\Gamma$ satisfies the source-sink condition, it admits a rotating-splitting circuit.
\end{lemma}

\begin{proof}
Assign to each vertex any rotating or splitting structure. This gives a partition of the edges of $\Gamma$ into edgewise disjoint cycles. If there is only one such cycle, we are done. If there is more than one cycle, since the graph is connected, there must be a vertex $v$ shared by different cycles. If $v$ has degree 4, it must be rotating, and there must be two cycles meeting at $v$. In this case we can join the cycles at $v$ and still have $v$ be rotating, see Figure \ref{fig:join_cycles_x}. If $v$ has degree 6, is rotating, and is the meeting point of three cycles, we can join the cycles and still have $v$ be rotating, as in Figure \ref{fig:join_3_cycles_rotating}. If $v$ has degree 6, is rotating, and is the meeting point of two cycles, we can join the cycles by appropriately changing $v$ into a splitting vertex, as in Figure \ref{fig:join_2_cycles_rotating}. If $v$ has degree 6, is splitting, and is the meeting point of three cycles, we can join the cycles and still have $v$ be splitting, as in Figure \ref{fig:join_3_cycles_splitting}. If $v$ has degree 6, is splitting, and is the meeting point of 2 cycles, we have two possibilities: The cycle which contains only two of the half-edges around $v$ may contain two adjacent half-edges around $v$, or it may contain two opposite half-edges around $v$. In both cases, we can join the cycles by appropriately changing $v$ into a rotating vertex, see Figure \ref{fig:join_2_cycles_splitting_flat} and Figure \ref{fig:join_2_cycles_splitting_crossed}. By iterating this process, we arrive at a rotating-splitting circuit of $\Gamma$.
\end{proof}

\begin{figure}
\centering
\includegraphics[trim = 3cm 18cm 0cm 5cm, clip, totalheight=4cm]{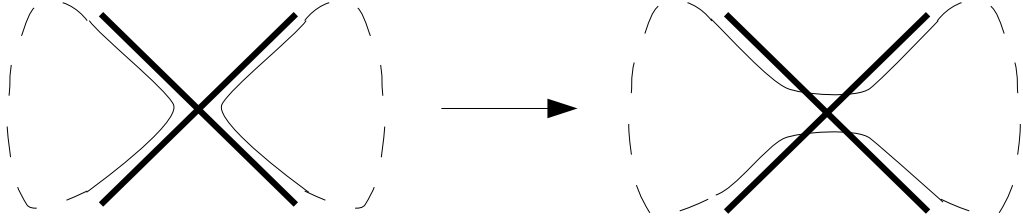}
\caption{Joining 2 cycles at a 4-vertex}
\label{fig:join_cycles_x}
\end{figure}

\begin{figure}
\centering
\includegraphics[trim = 0cm 21cm 0cm 0cm, clip, totalheight=4cm]{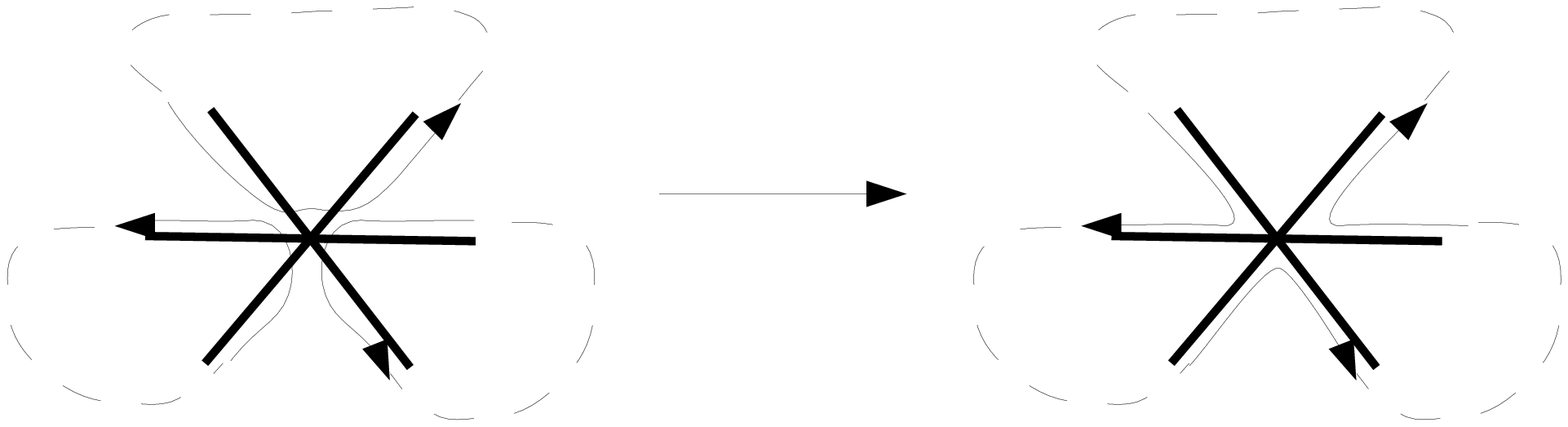}
\caption{Joining 3 cycles at a rotating 6-vertex}
\label{fig:join_3_cycles_rotating}
\end{figure}

\begin{figure}
\centering
\includegraphics[trim = 0cm 15cm 0cm 3cm, clip, totalheight=4cm]{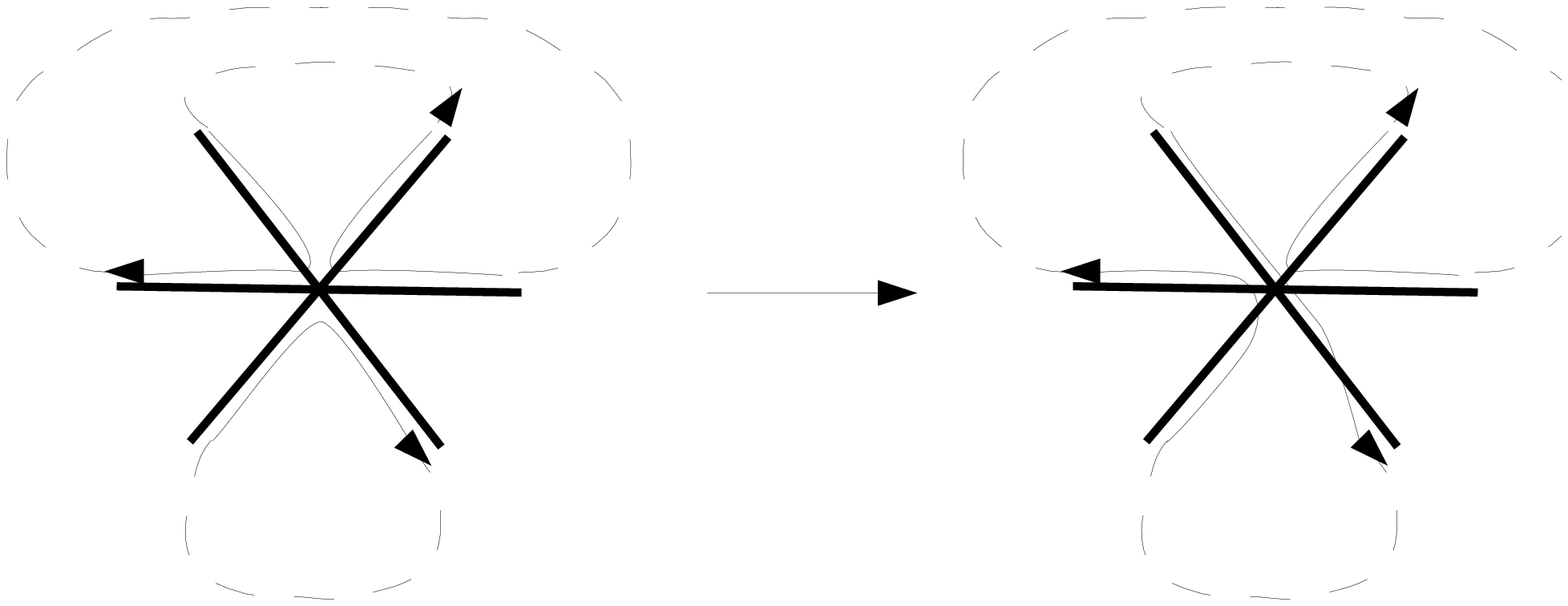}
\caption{Joining 2 cycles at a rotating 6-vertex}
\label{fig:join_2_cycles_rotating}
\end{figure}

\begin{figure}
\centering
\includegraphics[trim = 0cm 18cm 0cm 3cm, clip, totalheight=4cm]{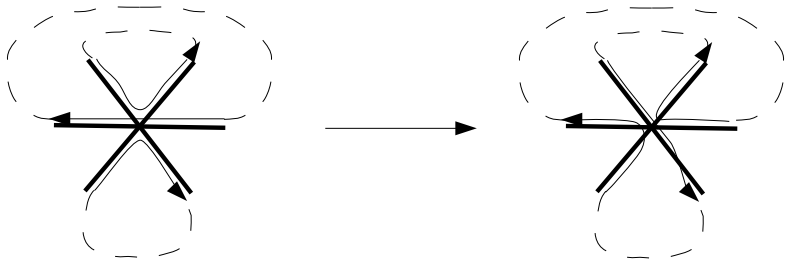}
\caption{Joining 3 cycles at a splitting 6-vertex}
\label{fig:join_3_cycles_splitting}
\end{figure}

\begin{figure}
\centering
\includegraphics[trim = 0cm 21cm 0cm 0cm, clip, totalheight=4cm]{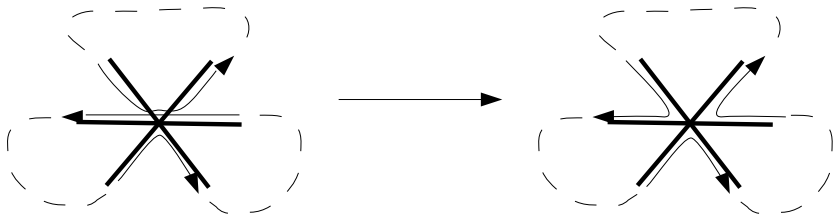}
\caption{Joining 2 cycles at a splitting 6-vertex (first case)}
\label{fig:join_2_cycles_splitting_flat}
\end{figure}

\begin{figure}
\centering
\includegraphics[trim = 0cm 15cm 0cm 5cm, clip, totalheight=4cm]{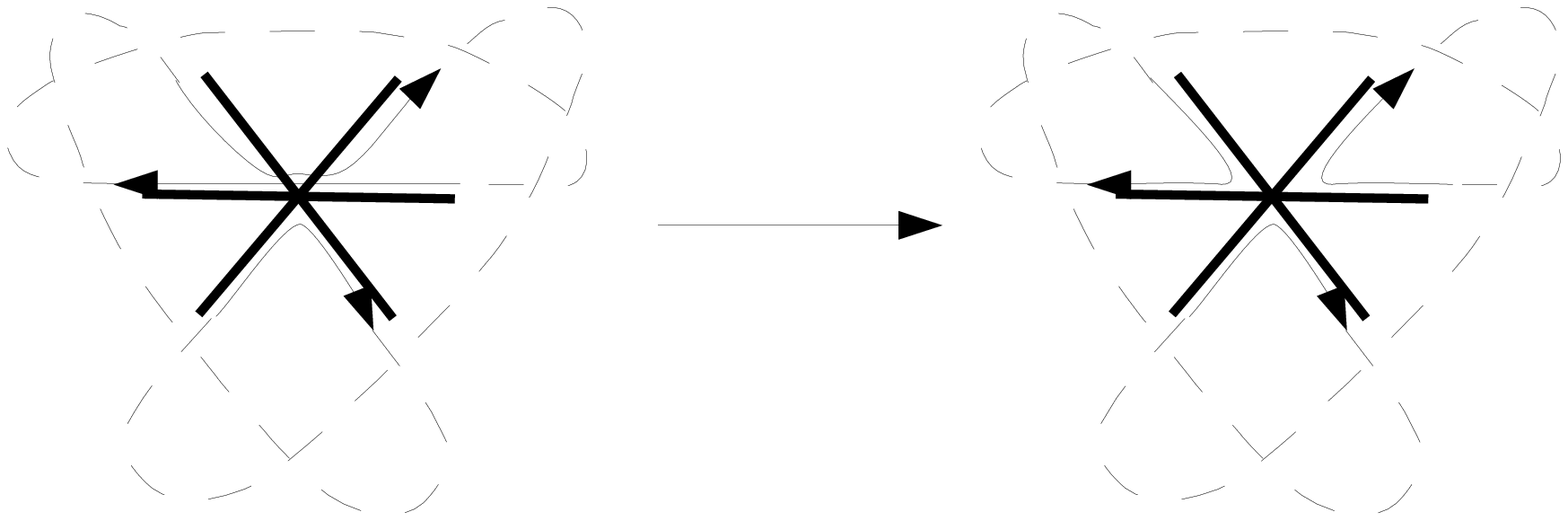}
\caption{Joining 2 cycles at a splitting 6-vertex (second case)}
\label{fig:join_2_cycles_splitting_crossed}
\end{figure}

For any $\st$-graph $\Gamma$ with all vertices of degree 4 or 6, any rotating-splitting circuit $C$ of $\Gamma$, and any rotating 6-vertex $v$ in $\Gamma$, observe that the points $C^{-1}(\{v\})$ divide $S^1$ into arcs, and the images of these arcs under $C$ are cycles in $\Gamma$, each of which contains two half-edges around $v$.

\begin{definition}
If each of these arcs contains two adjacent half-edges around $v$, $v$ is \emph{flat}. If each of these arcs contains two opposite half-edges around $v$, $v$ is \emph{crossed}.
\end{definition}

\begin{remark}
In a $\st$-graph with all vertices of degree 4 or 6 and which satisfies the source-sink condition, every rotating 6-vertex is either flat or crossed.
\end{remark}

\begin{figure}
	\begin{subfigure}{0.45\textwidth}
	\centering
	\includegraphics[trim = 2.5cm 16cm 0cm 3cm, clip=true, totalheight=4cm]{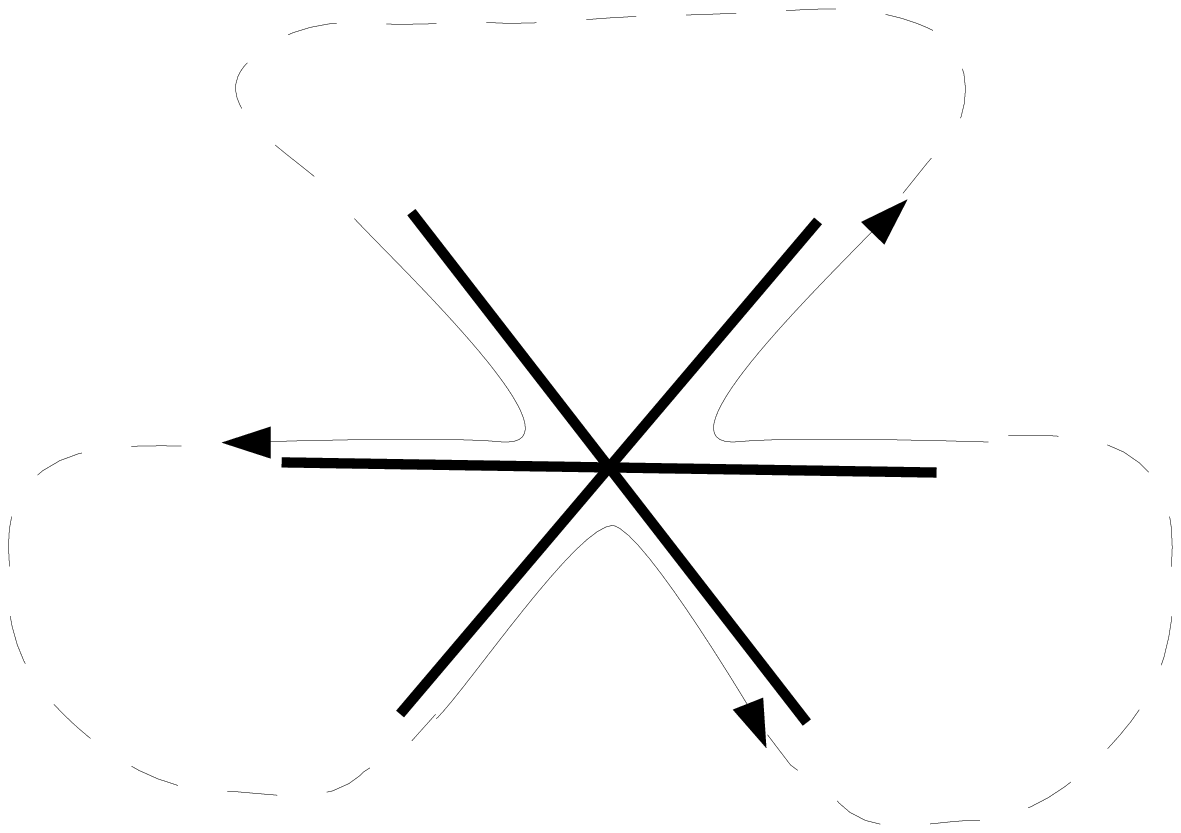}
	\caption{A flat rotating vertex}
	\end{subfigure}
	\quad
	\begin{subfigure}{0.45\textwidth}
	\centering
	\includegraphics[trim = 0cm 13cm 4cm 3cm, clip=true, totalheight=4cm]{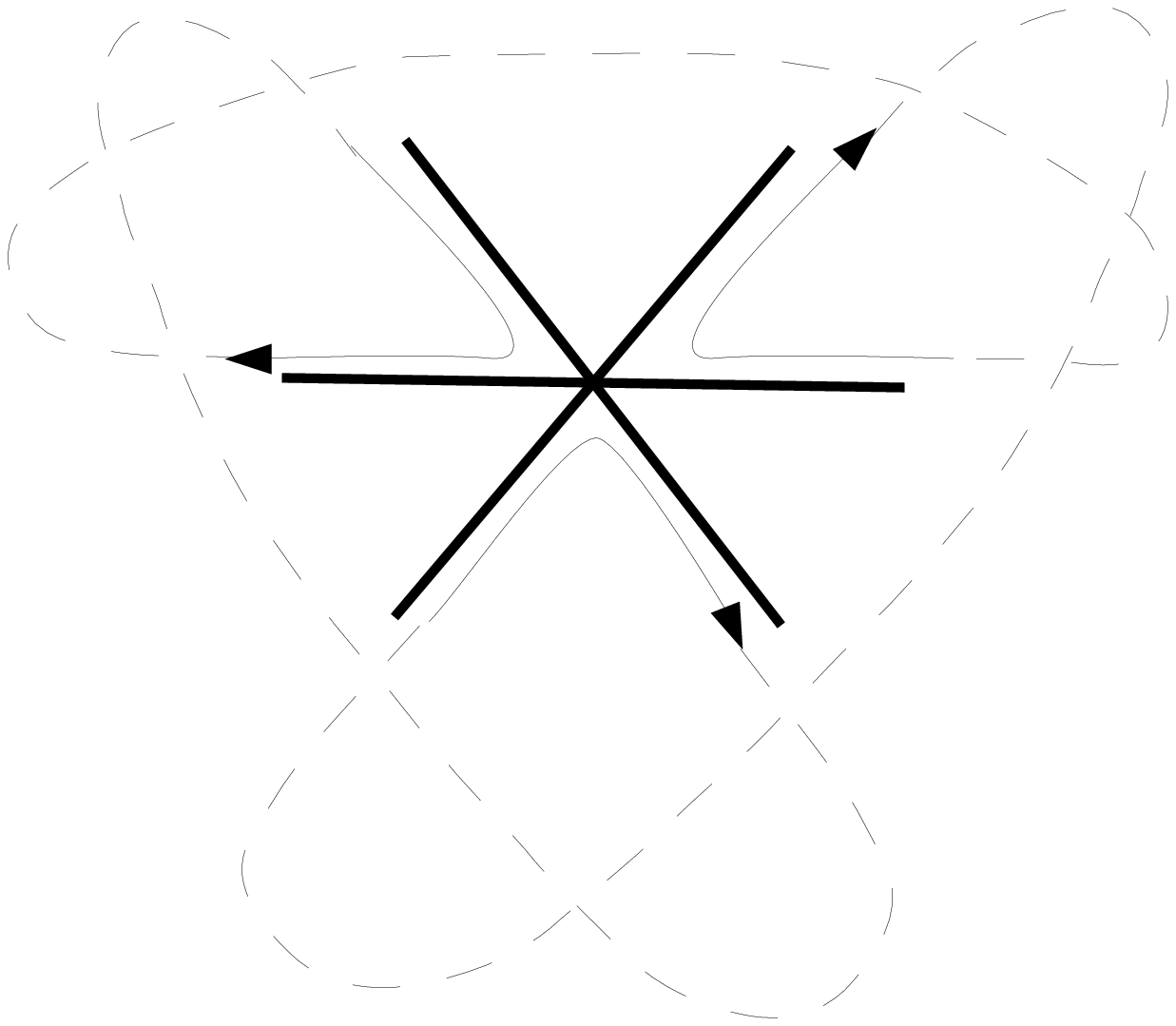}
	\caption{A  crossed rotating vertex}
	\end{subfigure}
	\caption{}
\end{figure}

\begin{definition}
A \emph{chord diagram} is a cubic graph with a distinguished simple cycle.
\end{definition}

\begin{definition}
A $\st$-chord diagram is a graph $D$ with a distinguished simple cycle, such that every vertex in $D$ has degree 3 or 4 and for every edge $e$ in $D$, one of the following holds:
\begin{enumerate}
\item $e$ is in the distinguished cycle.
\item Both of the vertices on $e$ are in the distinguished cycle, and both have degree 3.
\item One of the vertices on $e$ is in the distinguished cycle, the other is not, and both have degree 3.
\item Both of the vertices on $e$ are in the distinguished cycle, one has degree 3, and the other has degree 4.
\end{enumerate}
Additionally, the vertices not in the distinguished cycle are marked as flat or crossed.
\end{definition}

\begin{definition}
An \emph{arc} of a $\st$-chord diagram $\Gamma$ is an edge in the distinguished cycle of $\Gamma$.
\end{definition}

\begin{definition}
A \emph{chord} of a $\st$-chord diagram $\Gamma$ is an edge not in the distinguished cycle of $\Gamma$, connecting two vertices of degree 3 which are in the cycle.
\end{definition}

\begin{definition}
A \emph{triad} of a $\st$-chord diagram $\Gamma$ is a vertex $v$ not in the distinguished cycle of $\Gamma$, together with the three edges incident to $v$. The vertex $v$ is called a \emph{triad point}.
\end{definition}

\begin{remark}
Since triad points are the points not in the distinguished cycle, they are marked as flat or crossed. Triad points which are crossed are marked in diagrams with a black dot; flat triad points are unmarked.
\end{remark}

\begin{definition}
A \emph{double chord} of a $\st$-chord diagram $\Gamma$ is a pair of edges not in the distinguished cycle of $\Gamma$, which are incident to a shared vertex $v$. The vertex $v$ is called the \emph{principal vertex} of the double chord.
\end{definition}

\begin{figure}
\centering
\includegraphics[trim=0cm 18cm 10cm 1cm, clip=true, totalheight=4cm]{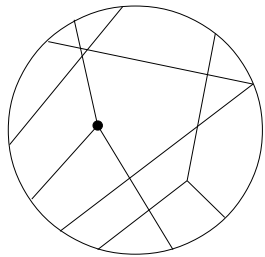}
\caption{a $*$-chord diagram with a crossed triad, a flat triad, a double chord, and a chord. Note that the only vertices on the interior of the circle are the two triad points; the other apparent vertices are simply the intersections of edges in the plane.}
\label{star-chord_diagram}
\end{figure}

\begin{remark}
Both in chord diagrams and in $\st$-chord diagrams, we can identify the distinguished cycle with $S^1$, and view the chords, triads, and double chords as additional structure on the points of $S^1$.
\end{remark}

\begin{definition}
Two chords $pq$ and $rs$ are \emph{linked} if $p$ and $q$ belong to different connected components of $S^1 \setminus\{r,s\}$.
\end{definition}

Given a $\st$-graph $\Gamma$ which meets the source-sink condition and with all vertices of degree 4 or 6, and given a rotating-splitting circuit $C$ of $\Gamma$, we define a $\st$-chord diagram $D_{\Gamma, C}$ as follows:

For any 4-vertex $v$ in $\Gamma$, the two points in $S^1$ which are mapped to $v$ by $C$ are connected by a chord. For any rotating 6-vertex $v$ in $\Gamma$, the three points in $S^1$ which are mapped to $v$ by $C$ are connected by a triad. For any splitting 6-vertex $v$ in $\Gamma$, the three points in $S^1$ which are mapped to $v$ by $C$ are connected by a double chord, whose principal vertex is $a \in C^{-1}(\{v\})$ such that $C(a-\epsilon)$ and $C(a+\epsilon)$ are in opposite half-edges around $v$.

A triad in $D_{\Gamma, C}$ is labeled as flat if the corresponding vertex in $\Gamma$ is flat; crossed if the corresponding vertex in $\Gamma$ is crossed.

\begin{definition}
An \emph{expansion} of a $\st$-chord diagram $D_{\Gamma, C}$ is a chord diagram $D'_{\Gamma, C}$, such that
\begin{enumerate}
\item For every chord $D_{\Gamma, C}$, containing vertices $a$ and $b$, there is a chord in $D'_{\Gamma, C}$ containing vertices $a$ and $b$.
\item For every triad $e$ in $D_{\Gamma, C}$, for some labeling $a,b,c$ of the vertices of $e$, $D'_{\Gamma, C}$ contains a chord connecting $a\pm\epsilon$ to $b$ and a chord connecting $a\mp\epsilon$ to $c$, with $\pm\epsilon$ chosen in such a way that the chords are linked if $e$ is crossed, and not linked if $e$ is flat.
\item For every double chord $e$ in $D_{\Gamma, C}$ with principal vertex $a$, for some labeling $b,c$ of the nonprincipal vertices of $e$, there is chord connecting $a-\epsilon$ to $b$ and a chord connecting $a+\epsilon$ to $c$. These chords are not linked.
\end{enumerate}
\end{definition}

\begin{remark}
Every $\st$-chord diagram has an expansion; but for $\st$-chord diagrams containing triads, the expansion is not unique. For our purposes it will not matter which expansion of the $\st$-chord diagram we take.
\end{remark}

\begin{figure}
\centering
\includegraphics[trim=0cm 18cm 10cm 1cm, clip=true, totalheight=4cm]{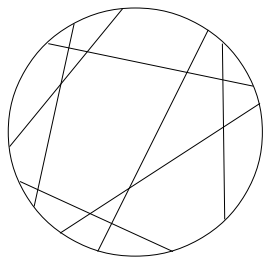}
\caption{An expansion of the $*$-chord diagram in Figure \ref{star-chord_diagram}}
\label{expansion}
\end{figure}

\begin{definition}
A \emph{permissible separation} of a chord diagram $D'_{\Gamma, C}$ arising as an expansion of a $\st$-chord diagram $D_{\Gamma, C}$ is a pair of chord diagrams $D_W$ and $D_B$ such that
\begin{enumerate}
\item Every chord in $D'_{\Gamma, C}$ is in exactly one of $D_W$ and $D_B$.
\item Two chords in $D'_{\Gamma, C}$ which come from the same triad in $D_{\Gamma, C}$ are both in $D_W$, or both in $D_B$.
\item Of any two chords in $D'_{\Gamma, C}$ which come from the same double chord in $D_{\Gamma, C}$, one chord is in $D_W$ and the other is in $D_B$.
\end{enumerate}
\end{definition}

\begin{figure}
\centering
\includegraphics[trim=0cm 18cm 0cm 1cm, clip=true, totalheight=4cm]{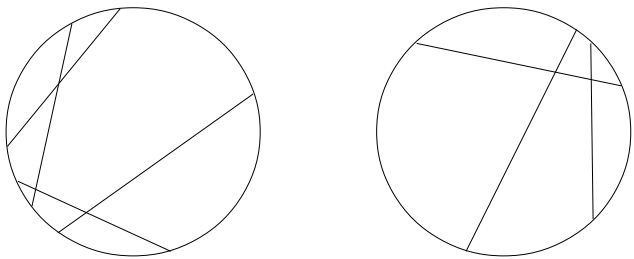}
\caption{A permissible separation of Figure \ref{expansion}}
\label{separation}
\end{figure}

Suppose $\Gamma$ is checkerboard-embedded in a closed surface $S$. Then the rotating-splitting circuit $C$ gives a mapping from $S^1$ to $S$ which is one-to-one except at the preimages of vertices $\Gamma$. This mapping can be smoothed to give an embedding of $S^1$ into $S$, as in Figure \ref{fig:smoothed_rscycle}. Observe that the circle $S^1 \subset S$ divides the surface into a black part and a white part. We can draw the chords of $D'_{\Gamma, C}$ as small edges lying in neighborhoods of vertices of $\Gamma$, see Figure \ref{fig:small_chord}. The coloring of $S$ divides the chords of $D'_{\Gamma, C}$ into two families: those lying in the white part and those lying in the black part. Observe that the two chords in the neighborhood of a rotating vertex are in the same part, and the two chords in the neighborhood of a splitting vertex are in different parts. Thus we have a permissible separation of $D'_{\Gamma, C}$.

\begin{figure}
\centering
\includegraphics[trim = 2cm 7cm 0cm 7cm, clip=true, totalheight=6cm]{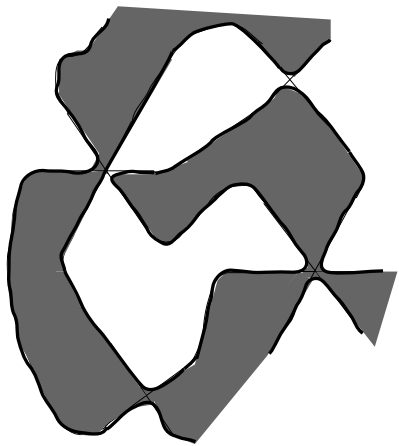}
\caption{The circle $S^1 \subset S$ divides the surface into a black part and a white part.}
\label{fig:smoothed_rscycle}
\end{figure}

\begin{figure}
	\caption{Chords drawn as small edges at:}
	\begin{subfigure}{0.3\textwidth}
	\caption{a 4-vertex}
	\centering
	\includegraphics[trim = 8cm 20cm 0cm 3cm, clip=true, totalheight=3cm]{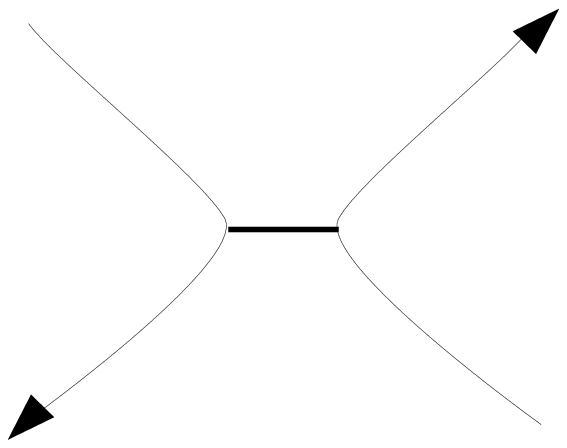}
	\end{subfigure}
	\quad
	\begin{subfigure}{0.3\textwidth}
	\caption{a rotating 6-vertex}
	\centering
	\includegraphics[trim = 5cm 17cm 4cm 5cm, clip=true, totalheight=3cm]{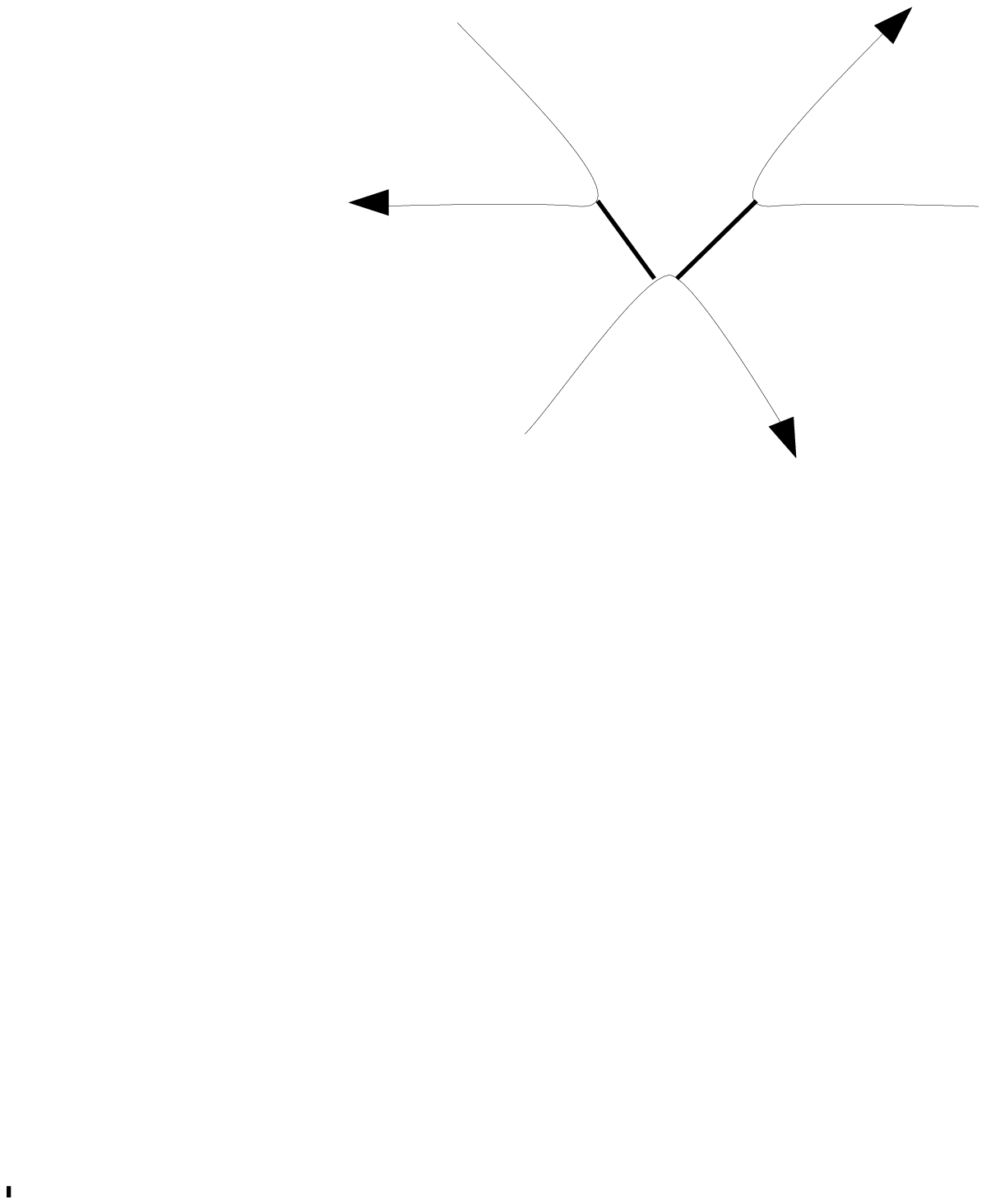}
	\end{subfigure}
		\quad
	\begin{subfigure}{0.3\textwidth}
	\caption{a splitting 6-vertex}
	\centering
	\includegraphics[trim = 6cm 17cm 4cm 4cm, clip=true, totalheight=3cm]{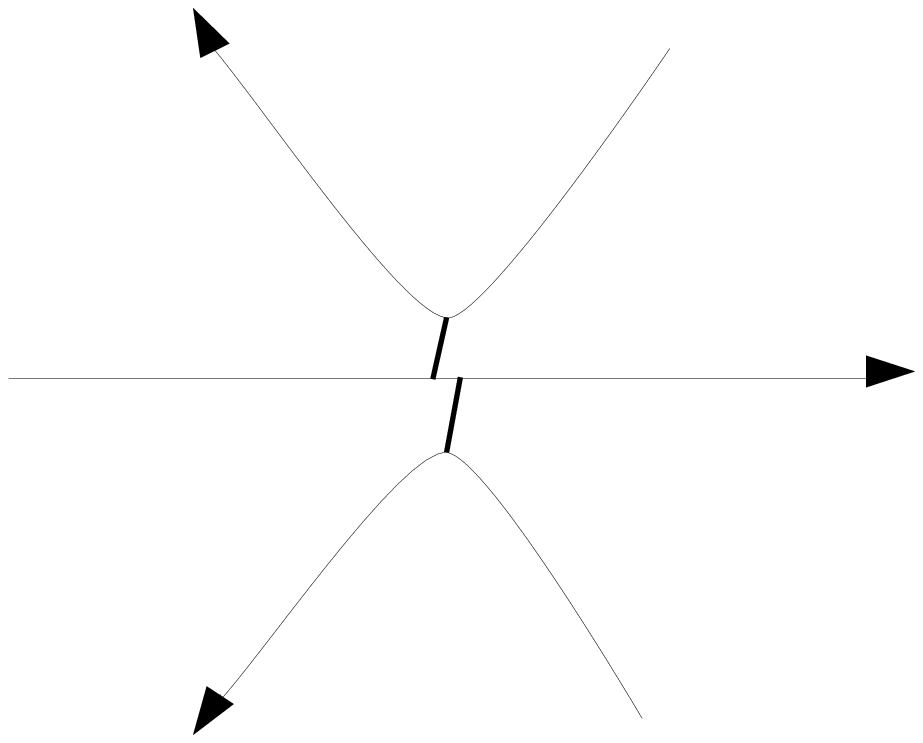}
	\end{subfigure}
	\label{fig:small_chord}
\end{figure}

Vice versa, given a $\st$-graph $\Gamma$ with all vertices of degree 4 or 6 and which satisfies the source-sink condition, a rotating-splitting circuit $C$ of $\Gamma$, and a permissible separation of $D'_{\Gamma, C}$, we can recover the coloring of the angles around each vertex of $\Gamma$, and thus we can recover the surface $S$. Thus given a $\st$-graph $\Gamma$ with all vertices of degree 4 or 6 and which satisfies the source-sink condition, and an expansion $D'_{\Gamma, C}$ of its $\st$-chord diagram, we have a one-to-one correspondence between atoms of $\Gamma$ and permissible separations of $D'_{\Gamma, C}$.

Note that regardless of whether a rotating vertex $v$ is flat or crossed, the two chords to be drawn in the neighborhood of $v$ do not cross in $S$, as shown in Figures \ref{flat_lifecycle} and \ref{crossed_lifecycle}. Thus we have an embedding of $D'_{\Gamma, C}$ into $S$. Furthermore, since the embedding of $\Gamma$ divides $S$ into 2-cells, the embedding of $D'_{\Gamma, C}$ does as well.

\begin{figure}
\includegraphics[trim = 1cm 20cm 0cm 3cm, clip=true, totalheight=3.3cm]{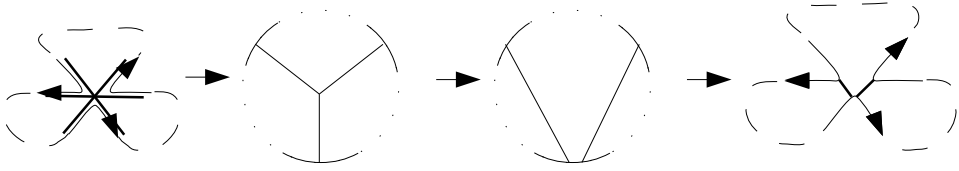}
\caption{The chords drawn in the neighborhood of a flat rotating 6-vertex do not cross.}
\label{flat_lifecycle}
\end{figure}

\begin{figure}
\includegraphics[trim = 1cm 20cm 0cm 3cm, clip=true, totalheight=3.3cm]{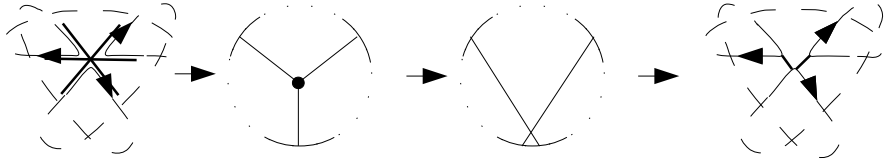}
\caption{The chords drawn in the neighborhood of a crossed rotating 6-vertex do not cross.}
\label{crossed_lifecycle}
\end{figure}

\begin{definition}
Given a chord diagram $D$, \emph{surgery} of $D$ is the following process: For each chord $e$ connecting points $a,b$, delete a neighborhood of $e$ and connect the obtained endpoints $a+\epsilon$ to $b-\epsilon$ and $a-\epsilon$ to $b+\epsilon$. This produces a family of circles; these are called the \emph{result of surgery} of $D$.
\end{definition}

\begin{figure}
\centering
\includegraphics[trim = 1cm 17cm 0cm 3cm, clip=true, totalheight=4cm]{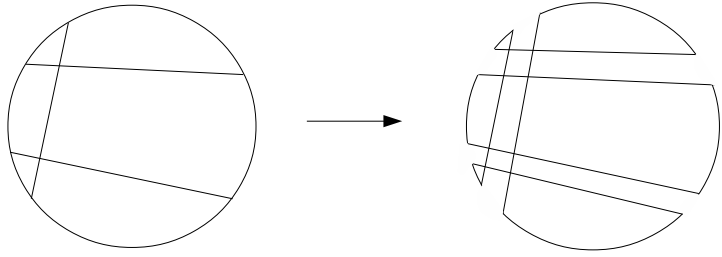}
\caption{Surgery of a chord diagram}
\label{surgery}
\end{figure}

\begin{definition}
To form the \emph{intersection matrix} of a chord diagram $D$ with $n$ chords, first enumerate the chords $1,\ldots,n$. Then the \emph{intersection matrix} $M(D)$ is an $n \times n$ matrix over $\mathbb Z_2$, such that $M_{ii}=0$ for any $i$, and $M_{ij}=M_{ji}=1$ for $i \neq j$ if and only if the chords $i$ and $j$ are linked.
\end{definition}

\begin{theorem}[Circuit-Nullity Theorem \cite{cohn},\cite{soboleva},\cite{traldi}]
The number of components in the manifold obtained from a chord diagram $D$ by surgery of the circle is one plus the corank of $M(D)$.
\end{theorem}

\begin{lemma}\label{sep-to-genus}
Given a $\st$-graph $\Gamma$ which satisfies the source-sink condition and in which all vertices have degree 4 or 6 and a rotating-splitting circuit $C$ of $\Gamma$, and a checkerboard embedding of $\Gamma$ into a surface $S$, the genus of $S$ is given by
\[(\mbox{rank}(M(D_W))+\mbox{rank}(M(D_B)))/2\]
where $D_W$ and $D_B$ are the results of the permissible separation of $D'_{\Gamma, C}$ induced by the embedding.
\end{lemma}
\begin{proof}
Consider the embedding $D'_{\Gamma, C}\to S$ described above.
The number of 2-cells on the white side of the embedding is the number of circles resulting in surgery of $D_W$. Likewise, the number of 2-cells on the black side of the embedding is the number of circles resulting in surgery of $D_B$. Applying the Circuit Nullity Theorem, the total number of 2-cells is $\mbox{corank}(M(D_W))+\mbox{corank}(M(D_B))+2$. Introducing the notation $|D'_{\Gamma, C}|$ to represent the number of chords in $D'_{\Gamma, C}$, the number of arcs in $D'_{\Gamma, C}$ is $2|D'_{\Gamma, C}|$, so its total number of edges is $3|D'_{\Gamma, C}|$. The number of vertices in $D'_{\Gamma, C}$ is $2|D'_{\Gamma, C}|$. Thus the Euler characteristic of $S$ is
\[ 2|D'_{\Gamma, C}| - 3|D'_{\Gamma, C}| + \mbox{corank}(M(D_W))+\mbox{corank}(M(D_B))+2\]
\[= -\mbox{rank}(M(D_W))-\mbox{rank}(M(D_B))+2 \]
so the genus of $S$ is $(\mbox{rank}(M(D_W))+\mbox{rank}(M(D_B)))/2$.
\end{proof}

Thus $\Gamma$ has admits an atom of genus $g$ if and only if some permissible separation of $D_{\Gamma, C}$ results in $\mbox{rank}(D_W)+\mbox{rank}(D_B)=2g$. This can be reduced to a problem on matrices, as follows.

\begin{definition}
A \emph{permissible partition} of the indices of $M(D'_{\Gamma, C})$ is a partition of the indices of $M(D'_{\Gamma, C})$ (which are just the chords of $D'_{\Gamma, C}$) into two parts, in such a way that chords arising from the same triad in $D_{\Gamma, C}$ are in the same part, and chords arising from the same double chord in $D_{\Gamma, C}$ are in different parts.
\end{definition}

Clearly, $D_W$ and $D_B$ are a permissible separation of $D'_{\Gamma, C}$ if and only if there exists a permissible partition $I \sqcup J$ of the indices of $M(D'_{\Gamma, C})$ such that $M(D_{\Gamma, C})_I$ is the intersection matrix of $D_W$ and $M(D_{\Gamma, C})_J$ is the intersection matrix of $D_B$.

\section{Main Result}

\begin{theorem}\label{main}
For a $\st$-graph $\Gamma$ with rotating-splitting circuit $C$, $\Gamma$ has a checkerboard embedding into an orientable surface of genus $g$ if and only if there is a permissible partition of the indices of $M(D_{\Gamma, C})$ into parts $I$ and $J$ such that $\mbox{rank}(M_I)+\mbox{rank}(M_J)=2g$.
\end{theorem}
\begin{proof}
Let $D'_{\Gamma, C}$ be any expansion of $D_{\Gamma, C}$.
By Lemma \ref{sep-to-genus}, $\Gamma$ has a checkerboard embedding into a surface of genus $g$ if and only there is a permissible separation $(D_W,D_B)$ of $D'_{\Gamma, C}$ such that $\mbox{rank}(M(D_W)) + \mbox{rank}(M(D_B))=2g$. Such a permissible separation exists if and only if there is a permissible partition of $M(D_{\Gamma, C})$ into parts $I$ and $J$ such that $\mbox{rank}(M_I)+\mbox{rank}(M_J)=2g$.
\end{proof}

Thus, the problem of finding the minimal genus into which a $\st$-chord diagram with each vertex of degree 4 or 6 and which satisfies the source-target condition may be checkerboard-embedded, is equivalent to the problem of finding a permissible partition of the indices of a matrix $M$ into parts $I$ and $J$ which minimizes $\mbox{rank}(M_I)+\mbox{rank}(M_J)$.

\section{The case of the plane}

A $\st$-graph $\Gamma$ with rotating-splitting circuit $C$ is embeddable into the plane if and only if there exists a permissible separation $(D_W,D_B)$ of $D'_{\Gamma, C}$ such that $\mbox{rank}(M(D_W)) = \mbox{rank}(M(D_B))=0$. In other words, $\Gamma$ is planar if and only if there exists a permissible separation of $D'_{\Gamma, C}$ into two chord diagrams, each of which has no linked chords. We can test this condition by the following algorithm, which takes time quadratic in the number of chords of $D'_{\Gamma, C}$, and thus quadratic in the number of vertices of $\Gamma$: Assign one of the chords to either $D_W$ or $D_B$. Then assign to the other part all chords linked to the first chord. If the first chord originates from a triad, assign the other chord coming from this triad to the same chord diagram, and if the first chord originates from a double chord, assign the other chord coming from this double chord to the other chord diagram.  Then, for each of the chords in this way, assign any unassigned linked chords or chords coming from the same triad or double chord, using the same rules described above. Repeat this process until for every assigned chord, the linked chords and any chord coming from the same chord diagram have been assigned. If not all chords have been assigned, take any unassigned chord and arbitrarily assign it to $D_W$ or $D_B$, and repeat until all chords have been assigned. Finally, check whether this is a permissible separation, and whether each of $D_W$ and $D_B$ contains no pairs of linked chords. $\Gamma$ is planar if and only if both of these conditions are true.

The cases when the source-sink condition does not hold or when the embedding is not checkerboard will be addressed in a future paper.

\end{document}